\documentclass[reqno]{amsart}
\usepackage[utf8]{inputenc}
\usepackage{amsmath, amssymb, amsthm, epsfig}
\usepackage{hyperref, latexsym}
\usepackage{url}
\usepackage[mathscr]{euscript}
\usepackage{mathrsfs}
\usepackage{ dsfont }
\usepackage{color}
\usepackage{fullpage} 
\usepackage{setspace}
\usepackage[left=3cm, right=3cm, bottom=3cm]{geometry}     
\usepackage{verbatim} 
\usepackage{bbm}

\theoremstyle{plain}
\newtheorem{theorem}{Theorem}
\newtheorem{corollary}[theorem]{Corollary}
\newtheorem{lemma}[theorem]{Lemma}

\theoremstyle{definition}
\newtheorem{definition}[theorem]{Definition}

\numberwithin{equation}{section}
\newtheorem*{theorem*}{Theorem} 

\newcommand{\Z}{{\mathbb Z}}
\newcommand{\R}{{\mathbb R}}
\newcommand{\N}{{\mathbb N}}
\newcommand{\C}{{\mathbb C}}

\def\XXint#1#2#3{{\setbox0=\hbox{$#1{#2#3}{\int}$}
\vcenter{\hbox{$#2#3$}}\kern-.5\wd0}}

\newcommand{\dx}{\, \mathrm{d}x}

\newcommand{\pa}[1]{\left(#1\right)}
\newcommand{\abs}[1]{\left\vert#1\right\vert}
\newcommand{\suc}[1]{\left\{#1\right\}}
\newcommand{\set}[2]{\left\{#1 :#2\right\}}

\renewcommand{\dx}{\,\text{\rm d}x}

\DeclareMathOperator{\essinf}{ess\,inf}

\providecommand{\abs}[1]{ \lvert#1  \rvert}
\providecommand{\norm}[1]{ \lVert#1  \rVert}

\newcommand{\sinc}{\mathrm{sinc}}

\title{Dynamical sampling, derivatives, and interpolation formulas in the Paley--Wiener space}

\author[Gardeazabal]{Iker Gardeazabal-Gutiérrez}
\author[Sousa]{Mateus Sousa}
\address{BCAM - Basque Center for Applied Mathematics, Alameda de Mazarredo 14, 48009 Bilbao, Bizkaia, Spain}
\email{iker.gardeazabal@ehu.eus} 
\email{mcosta@bcamath.org}


\allowdisplaybreaks

\begin{document}

\begin{abstract}  
In this paper we present a criteria to obtain interpolation formulas  in terms of the sequence $\pa{\suc{T_n(f)(Nm)}_{m\in\Z}}_{n=1}^N$, where $f$ is a function whose Fourier transform is supported in $[-1/2,1/2]$, and $T_n$ are certain Fourier multiplier operators. We also discuss applications and also prove that our results recover several classical formulas. 
\end{abstract}

\maketitle

\section{Introduction}

Recovering bandlimited functions by sampling their values over a discrete set is a classical problem in analysis. One of the most basic results in this direction is the Shannon-Whittaker interpolation theorem for Paley-Wiener spaces. We denote by $PW_{2\pi\delta}$ the Paley-Wiener space of type $2\pi\delta$, which consists of the functions $f\in L^2(\R)$ whose Fourier transform\footnote{Here we normalize the Fourier transform by $\widehat{f}(\xi)=\int_\R f(x)e^{-2\pi ix\cdot \xi}\dx$} $\widehat{ f}$ is supported in $ [-\delta,\delta]$. The Shannon-Whittaker theorem says for every $f\in PW_\pi$ and $x\in \R$ one has
\begin{align}\label{eq:12/07/2018 01:32}
    f(x)=\sum_{n\in \Z}f(n)\sinc(x-n),
\end{align}
where 
$$\sinc(x)=\frac{\sin(\pi x)}{\pi x},$$
is the cardinal sine function.

The series in \eqref{eq:12/07/2018 01:32} has two interpretations. The limit happens in the $L^2$ sense, i.e, 
\begin{align*}
    \lim_{N\rightarrow\infty}\|f-\sum_{|n|\leq N}f(n)\sinc(\cdot-n)\|_2= 0.
\end{align*}
In fact, the set of functions $\{\sinc(\cdot-n)\}_{n\in\Z}$ forms an orthonormal system, in particular this implies
\begin{align*}
    \|f\|_2^2=\sum_{n\in\Z}|f(n)|^2.
\end{align*}
On the other hand, by the Paley-Wiener theorem, $PW_{2\pi\delta}$ is equivalent to the space of entire functions of exponential type\footnote{We say a entire function $f$ has exponential type $\tau$ if $\limsup_{|z|\rightarrow\infty}\tfrac{\log|f(z)|}{|z|}\leq \tau$} $2\pi\delta$ whose restriction to $\R$ is in $L^2(\R)$, and one has  
\begin{align}\label{eq:12/07/2018 02:32}
        \lim_{N\rightarrow\infty}{\sup_{z \in K}\left|f(z)-\sum_{|n|\leq N}f(n)\sinc(z-n)\right|}=0\, ,
\end{align}
for any compact set $K\subset \C$. This means that one can completely rebuild a function with Fourier transform supported in $[-1/2,1/2]$ by knowing its values over $\Z$.

By a simple dilation argument, one can see that for functions in $PW_{2\pi\delta}$ a similar conclusion holds: one can completely recover a function by knowing its values over $\tfrac{1}{2\delta}\Z$. This means that if a function has a bigger exponential type, one needs a denser dilation of $\Z$ to reconstruct the functions. In particular, for $PW_{2\pi}$ one needs the information over $\tfrac{1}{2}\Z$ to determine the functions, i.e., by doubling the exponential type, one doubles the information needed. In \cite{Va1985}, Vaaler studied the problem of determining a function $f\in PW_{2\pi}$ by knowing $\{(f(n),f'(n))\}_{n\in\Z}$ and obtained the following interpolation result: 
if $f\in PW_{2\pi}$ then 
\begin{align}\label{eq:Vaaler_limit}
    f(z)=\frac{1}{\pi^2}\sum_{n\in\Z}f(n)\frac{(\sin(\pi z))^2}{(z-n)^2}+\sum_{n\in\Z}f'(n)\frac{(\sin(\pi z))^2 }{(z-n)},
\end{align}
where the expression in the right-hand side of \eqref{eq:Vaaler_limit} converges uniformly in compact sets of $\C$ (see \cite[Theorem 9]{Va1985}).

This means one can reconstruct a $L^2$ function of exponential type $2\pi$ by replacing half the information needed about the values of $f$ with information about the values of $f'$, and this recovers the function in the pointwise sense. Although the system of interpolating functions in \eqref{eq:Vaaler_limit} is not orthogonal as in the  case of integer translations of the cardinal sine, it was later observed by Gonçalves \cite[Corollary 2]{G2015} that there are constants $c,C>0$ such that
\begin{align}\label{eq:12/07/2018 14:20}
    c\int_\R |f(x)|^2\dx \leq\sum_{n\in\Z}|f(n)|^2+\sum_{n\in\Z}|f'(n)|^2 \leq C\int_\R |f(x)|^2\dx,
\end{align}
for every $f\in PW_{2\pi}$. This implies the limit in \eqref{eq:Vaaler_limit} also happens in the $L^2$ sense.

In \cite{Li2003}, Littmann obtained a more general version of this kind of interpolation formula where for any $f\in PW_{N\pi}$, $N\in\N$, the following formula holds
\begin{equation}\label{Lit1}
    f(x)=\sum_{k=0}^{N-1} \sum_{m\in\Z} f^{(k)}(m) P^N_{N-k-1}\pa{\pi(x-m)} \frac{(x-m)^k}{k!} \pa{\sinc(x-m)}^N,
\end{equation}
where $P^N_n$ is the Taylor polynomial of degree $n$ of $\pa{\frac{x}{\sin(x)}}^N$ centered in 0. As in the Vaaler case, these formulas are accompanied by a norm equivalence, in this case, Gonçalves and Littmann obtained in \cite{GL2015} that there exist constants $c_N,C_N>0$ such that for every $f\in PW_{N\pi}$
\begin{equation}\label{Lit2}
    c_N\int_\R |f(x)|^2\dx \leq\sum_{k=0}^{N-1}\sum_{m\in\Z}|f^{(k)}(m)|^2 \leq C_N\int_\R |f(x)|^2\dx.
\end{equation}
Results like that are examples of the so-called dynamical sampling. Given a separable Hilbert space $H$, an operator $A:H\rightarrow H$, a family of vectors $\{f_n\}_{n\in I}$, where $I\subset \Z$, and a 
sequence $\{l_n\}_{n\in I}$ where each $l_n\in \Z_+$, a dynamical sampling question consists of asking if the collection $\{A^{i}f_n,\,i=0,\cdots,l_n\}_{n\in I}$ forms a complete system, a Riesz basis, a frame, or a Bessel sequence; see \cite{ACMT2014,ADK2013,CMPP2017} and the references within for more on the topic of dynamical sampling. 

If one considers $A=\frac{d}{dx}:PW_{N\pi}\rightarrow PW_{N\pi}$, and take $f_n^N(x)=N\sinc(N(x-n))$, 
one has
$$(-1)^j f^{(j)}(n) = (-1)^j \langle A^j f, f^N_n \rangle =  \langle f, A^j f^N_n \rangle.$$
Therefore equation \eqref{Lit2} implies the family of functions $\{A^j f^N_n\}$, with $j=0,...,N-1$ 
and $n\in\Z$, forms a frame. Moreover, it follows from the interpolation formula \eqref{Lit1} that 
$\{A^j f^N_n\}$ also forms a Riesz basis. So questions about interpolation and sampling of functions using the values of its derivatives in Paley-Wiener spaces are essentially dynamical sampling problems. 

{Interpolation formulas are a classical and useful tool for optimization problems (see \cite{Va1985}), and recently have garnered new interest due to the breakthroughs in sphere packing by Maryna Viazovska and her co-authors \cite{CKMRV17,Via17}, and the subsequent developments involving sign uncertainty principles (see \cite{CG19}) and universal optimality (see \cite{CKMRV22}). We refer to the survey \cite{Cohn24} for an account of the recent developments in Fourier interpolation.}

\subsection{Main results}

The aim of this paper is to generalize the interpolation formulas mentioned in the introduction. To 
be more precise, we will see under which conditions an interpolation formula for $f\in PW_\pi$ depending on 
$\pa{\suc{T_n(f)(Nm)}_{m\in\Z}}_{n=1}^N$ is possible or not, 
where $N\in\N$ and the $T_n$ are Fourier multiplier operators, meaning that
$$T_n(f)(x)=\int_{-\infty}^\infty \widehat{f}(\xi) K_n(\xi) e^{2\pi i x \xi} d\xi.$$
In what follows we will assume that the Fourier multipliers $K_n$ are essentially bounded in the interval $[-1/2,1/2]$, hence these operators send the space $PW_{\pi}$ into itself. Our main result is a criteria for such formulas to hold in terms of invertibility of certain matrices.   


\begin{theorem}\label{theo}
    Let $N\in\N$ and $T=\pa{T_n}_{n=1}^N$ where, for each $n=1,...,N$, $T_n$ is a Fourier multiplier operator associated 
    to a bounded  multiplier $K_n$. Let us define for $x\in\pa{\frac{N-2}{2},\frac{N}{2}}$ the $N \times N$ matrix
	$$\mathcal{M}_T(x)=\pa{\frac{1}{N} K_n\pa{\frac{m-1-x}{N}}}_{m,n=1}^N.$$
	Then the following are equivalent:
	
%
    1) $\essinf_{x\in \pa{\frac{N-2}{2},\frac{N}{2}}} \abs{\det \mathcal{M}_T(x)}>0$,
	
	2) There exists, for $n=1,...,k$, functions $g^T_n\in PW_\pi$ such that
	$$f(x)=\sum_{n=1}^N \sum_{m\in\Z} T_n(f)(Nm) g^T_n(x-Nm),$$
	for all $f\in PW_\pi$, where the convergence holds both in $L^2$ sense and uniformly.
	
	Furthermore, the functions $g^T_n$ appearing in 3) can be uniquely determined by either of the following 
	characterizations:
	
	i) For each $n=1,...,N$, the function $g^T_n$ is the unique function in $PW_\pi$ satisfying
	$$T_m(g^T_n)(Nj)=\delta_{n,m}\delta_{j,0},$$
	for all $m=1,...,N$ and $j\in\Z$,
	
	ii) The functions $g^T_n$, $n=1,...,N$, are the unique functions in $PW_\pi$ satisfying
	$$\pa{\widehat{g^T_n}\pa{\frac{j-1-\xi}{N}}}_{n,j=1}^N=\pa{\mathcal{M}_T(\xi)}^{-1},$$
	for almost every $\xi\in\pa{\frac{N-2}{2},\frac{N}{2}}$.
\end{theorem}
As in the aforementioned interpolation formulas available in the literature, a standard application of the open mapping theorem also reveals the following frame bounds.
\begin{corollary} \label{cor}
    Let $N\in\N$ and $T=\pa{T_n}_{n=1}^k$ where, for each $n=1,...,N$, $T_n$ is a Fourier multiplier operator associated 
    to a bounded multiplier $K_n$. Then, if any of the equivalent statements 1), 2) or 3) from Theorem \ref{theo} holds 
    then there exists two positive constants $c_T$ and $C_T$ such that
    \begin{equation*}
    c_T \int_\R \abs{f(x)}^2 dx \leq \sum_{n=1}^N \sum_{m\in\Z} \abs{T_n(f)(Nm)}^2
    \leq C_T \int_\R \abs{f(x)}^2 dx,    
    \end{equation*} 
    for all $f\in PW_\pi$.
\end{corollary}

\subsection{Stable sampling and interpolation set}

We will now consider a more general case, where it will not be possible to obtain an interpolation formula, but we will be able to get something weaker that still maintains some of its characteristics, to be more precise we consider the following definition. 

\begin{definition}
	Let $T=\pa{T_n}_{n=1}^N$ be a collection of $N$ Fourier multiplier operators and $\rho>0$. We say that $\pa{\rho\Z,T}$ is a set of stable sampling for $PW_{\delta\pi}$ if
	$$\norm{f}_{L^2(\R)}^2 \lesssim \sum_{n=1}^N \sum_{k\in\Z} \abs{T_n(f)(\rho k)}^2 \lesssim \norm{f}_{L^2(\R)}^2$$
	for all $f\in PW_{\delta\pi}$. We say that $\pa{\rho\Z,T}$ is an interpolation set for $PW_{\delta\pi}$ if the interpolation problem
	$$T_n(f)(\rho k)=a^n_k, \quad n=1,...,N, \ k\in\Z$$
	has a solution on $PW_{\delta\pi}$ for any $(\{a^n_k\}_{k\in\Z})_{n=1}^N\in \pa{\ell^2}^N$.
\end{definition}

We will now present a result that aims to see when $(\rho\Z,T)$ is, or is not, a set of stable sampling or interpolation set for $PW_{\delta\pi}$, using as a condition the values of the determinate of a matrix of functions, in a similar way the the main theorem. It should be noted that this result is a generalization of Theorem 1.6 in \cite{GS}, where the authors considered this problem in the case where the Fourier multiplier operators considered are of the form $T_{nL+s}(f)(x)=f^{(n-1)}(x_m+x)$, for $m=1,...,L$ and $n=1,....,r$.

\begin{theorem} \label{theo2}
	Let $N\in\N$, $\delta,\rho>0$ and $T=\pa{T_n}_{n=1}^N$ where, for each $n=1,...,N$, $T_n$ is a Fourier multiplier operator associated to the bounded multiplier $K_n$. Let us define for $x\in\pa{\frac{N-2}{2},\frac{N}{2}}$ the $N\times N$ matrix
	$$M^\rho_T(x):=\pa{\frac{1}{N} K_n\pa{\frac{m-1-x}{\rho}}}_{n,m=1}^N.$$
	Then we have that: \\
	
	1) If $\essinf_{x\in\pa{\frac{N-2}{2},\frac{N}{2}}} M^\rho_T(x)>0$ then:
	
	$\quad \bullet$ $\pa{\rho\Z,T}$ is a set of stable sampling for $PW_{\delta\pi}$ if and only if $\rho\delta\leq N$,
	
	$\quad \bullet$ $\pa{\rho\Z,T}$ is an interpolation set for $PW_{\delta\pi}$ if and only if $\rho\delta\geq N$.
    	
	2) If $\essinf_{x\in\pa{\frac{N-2}{2},\frac{N}{2}}} M^\rho_T(x)=0=\abs{\set{x\in\pa{\frac{N-2}{2},\frac{N}{2}}}{M^\rho_T(x)=0}}$ then:

    $\quad \bullet$ $\pa{\rho\Z,T}$ is a set of stable sampling for $PW_{\delta\pi}$ if $\rho\delta\leq N$,
	
	$\quad \bullet$ $\pa{\rho\Z,T}$ is not an interpolation set for $PW_{\delta\pi}$ if $\rho\delta\leq N$. 
    
	3) If $\abs{\set{x\in\pa{\frac{N-2}{2},\frac{N}{2}}}{M^\rho_T(x)=0}}> 0$ then:

    $\quad \bullet$ $\pa{\rho\Z,T}$ is not a set of stable sampling for $PW_{\delta\pi}$ if $\rho\delta\geq N$, 
	
	$\quad \bullet$ $\pa{\rho\Z,T}$ is not an interpolation set for $PW_{\delta\pi}$ if $\rho\delta\leq N$. 
\end{theorem}

\section{Proofs of Theorems}

\subsection{Preliminary lemmas} We shall carry out 2 proofs of Theorem \ref{theo}, and we begin by presenting two auxiliary lemmas which will be useful for both proofs. The first lemma gives us necessary and sufficient conditions for the injectivity and surjectivity of operators defined by matrices of bounded functions.

\begin{lemma}\label{lem1}
	Let $M$ be a $N\times N$ matrix whose entries are $L^\infty(a,b)$ functions, and its associated linear operator
	\begin{align*}
		O_M: \pa{L^2(a,b)}^N &\to \pa{L^2(a,b)}^N \\
		(f_n)_{n=1}^N &\to M (f_n)_{n=1}^N.
	\end{align*}
	Then we have that: 
	
	- $O_M$ is injective if and only if $M(x)\neq 0$ for almost every $x\in (a,b)$,
	
	- $O_M$ is surjective if and only if $\essinf_{x\in (a,b)}\abs{\det M(x)}>0$,
	
	- $O_M$ is invertible if and only if $\essinf_{x\in (a,b)}\abs{\det M(x)}>0$.
\end{lemma}

\begin{proof}[Proof of Lemma \ref{lem1}]
	For this proof we will use that the injectivity and surjectivity of an operator is equivalent to the existence of a left and right inverse, respectively. It is clear the the only candidate for a left or right inverse is the operator associated to the matrix of functions that is equal to the inverse of $M$ for almost every $x\in (a,b)$. Due to this, the condition $M(x)\neq 0$ for almost every $x\in (a,b)$ is necessary for the existence of the left and right inverse. Note that this ``inverse", without further assumptions, will send $O_M(L^2(a,b))$ into $L^2(a,b)$, which is enough for it to be the left inverse, therefore proving the sufficiency of the condition for the first point. In the case of the right inverse, we will need it to send $L^2(a,b)$ into itself or in other words that the entries of its associated matrix are functions in $L^\infty(a,b)$, which is equivalent to $\essinf_{x\in (a,b)}\abs{\det M(x)}>0$. Lastly, the third point is a consequence of the first two. \\
\end{proof}
The second lemma introduces us to some periodic functions that will be of great use in both proofs of the theorem.
\begin{lemma}\label{lem2}
    Let $T$ be a Fourier multiplier operator associated to a bounded and measurable multiplier $K$, $f$ a function 
    in $PW_\pi$, $N\in\N$ and $a\in\C$. Then we have that
	$$\sum_{m\in\Z} T(f)(Nm+a) e^{2\pi i m x}
	=\frac{1}{N} \sum_{m=1}^N \widehat{f}\pa{\frac{m-1-x}{N}} K\pa{\frac{m-1-x}{N}}e^\frac{2\pi i a (m-1-x)}{N},$$
	for $x\in \pa{\frac{N-2}{2},\frac{N}{2}}$.
\end{lemma}

\begin{proof}[Proof of Lemma \ref{lem2}]
    Let us start by defining $g(x)=\widehat{f}\pa{-\frac{x}{N}} K\pa{-\frac{x}{N}} e^{-\frac{2\pi i a x}{N}}$ 
    that satisfies
    \begin{align*}
		\widehat{g}(\xi) &=\int_{-\infty}^\infty g(x) e^{-2\pi i x \xi} dx
        =\int_{-\infty}^\infty \widehat{f}\pa{-\frac{x}{N}} 
        K\pa{-\frac{x}{N}} e^{-\frac{2\pi i a x}{N}} e^{-2\pi i x \xi} dx \\
        &=N \int_{-\infty}^\infty \widehat{f}(y) K(y) e^{2\pi i y (N\xi+a)} dy
        =N T(f)(N\xi+a).
	\end{align*}
    Using this function we have that if $x\in \pa{\frac{N-2}{2},\frac{N}{2}}$
	\begin{align*}
		\sum_{m\in\Z} T(f)(Nm+a) e^{2\pi i m x} &=\frac{1}{N} \sum_{m\in\Z} \widehat{g}(m) e^{2\pi i m x}
        =\frac{1}{N} \sum_{m\in\Z} g(x+m) \\
        &=\frac{1}{N} \sum_{m\in\Z} \widehat{f}\pa{-\frac{x+m}{N}} K\pa{-\frac{x+m}{N}} e^{-\frac{2\pi i a (x+m)}{N}} \\
        &=\frac{1}{N} \sum_{m=1}^N \widehat{f}\pa{\frac{m-1-x}{N}} K\pa{\frac{m-1-x}{N}}e^\frac{2\pi i a (m-1-x)}{N},
	\end{align*}
    where in the second equality we applied Poisson's summation formula and in the last one we used that 
    $\widehat{f}$ is supported on $[-1/2,1/2]$. \\
\end{proof}

\subsection{First proof of Theorem \ref{theo}}

The first proof will be based on the idea of obtaining a bijective relation between the data of the desired formula, $\pa{\suc{T_n(f)(Nm)}_{m\in\Z}}_{n=1}^N$, with the data of an already known formula, in this case we will use the Shannon-Whittaker interpolation formula, whose data is $\suc{f(m)}_{m\in\Z}$. {This idea was already explored in \cite[Theorem 1.2]{RS2020}, where the authors present an alternative proof of the aforementioned interpolation formula obtained by Vaaler in \cite{Va1985}. Note that, with this reasoning, we transform the problem of the existence of an interpolation formula into a problem of inverting an operator between sequence spaces, which can be quite complex.} In order to illustrate this reasoning we will first give a simplified version of it, where we will just change one value in the data of the interpolation formula. \\

We consider the problem of the existence of an interpolation formula on $PW_\pi$ only depending on $\suc{f(m)}_{m\in\Z\setminus\{0\}}$ and $f'(a)$ for some $a\in\R$, which is simply the Shannon--Whittaker interpolation formula but changing $f(0)$ by $f'(a)$. We start by using the Shannon--Whittaker interpolation formula to put $f'(a)$ as a function of $\{f(m)\}_{m\in\Z}$ as follows
$$f'(a)=\sum_{m\in\Z} f(m) \sinc'(a-m),$$
from this, if $\sinc'(a)\neq 0$ we can obtain that
$$f(0)=f'(a)\frac{1}{\sinc'(a)}-\sum_{m\in\Z\setminus\{0\}} f(m) \frac{\sinc'(m-a)}{\sinc'(a)}.$$
Therefore if $\sinc'(a)\neq 0$ we have obtained a way to write $f(0)$ as a linear combination  of the values of the interpolation formula we want to obtain, using this we get
\begin{align*}
    f(x) &= \sum_{m\in\Z} f(m) \sinc(x-m) \\
    &= \pa{f'(a)\frac{1}{\sinc'(a)}-\sum_{m\in\Z\setminus\{0\}} f(m)\frac{\sinc'(m-a)}{\sinc'(a)}} 
    \sinc(x) + \sum_{m\in\Z\setminus\{0\}} f(m) \sinc(x-m) \\
    &= f'(a)\frac{\sinc(x)}{\sinc'(a)} + 
    \sum_{m\in\Z\setminus\{0\}} f(m) \pa{\sinc(x-m)-\frac{\sinc'(m-a)}{\sinc'(a)}\sinc(x)},
\end{align*}
obtaining the interpolation formula we were looking for. Note that the condition $\sin'(a)\neq 0$ is not only sufficient for the formula to exist but also necessary as if this is not the case then any multiple of the $\sinc$ function will have all the data in the interpolation formula equal to 0. 

\begin{proof}[First proof of Theorem \ref{theo}:]
    We start by considering the operator $A_T: \pa{\ell^2}^N \to \pa{\ell^2}^N$ defined as
    $$A_T\pa{\pa{\left\{ b^n_m \right\}_{m\in\Z}}_{n=1}^N}
	=\pa{\left\{ \sum_{s=1}^N\sum_{j\in\Z} b_j^s T_n\pa{\sinc}(N(m-j)-s+1)\right\}_{m\in\Z}}_{n=1}^N,$$
	note that if $f\in PW_\pi$ we have that
	$$A_T\pa{\pa{\left\{ f(Nm+n-1) \right\}_{m\in\Z}}_{n=1}^N}
	=\pa{\left\{ T_n(f)(Nm)\right\}_{m\in\Z}}_{n=1}^N.$$
	We also consider the operator $F_N:L^2\pa{\R / \Z,\R^N}\to \pa{\ell^2}^N$ defined as
	$$F_N\pa{\pa{f_m}_{m=1}^N}=\pa{\left\{\widehat{f_m}(j)\right\}_{j\in\Z}}_{m=1}^N$$
	that is invertible, with inverse $F_N^{-1}:\pa{\ell^2}^N \to L^2\pa{\R / \Z,\R^N}$
	$$F_N^{-1}\pa{\pa{\left\{a^m_j\right\}_{j\in\Z}}_{m=1}^N}(x)=\pa{\sum_{j\in\Z} a^m_j e^{2\pi i j x}}_{m=1}^N.$$
	Now we compose these operators as follows
	\begin{align*}
		&\pa{F_N^{-1} \circ A_T \circ F_N}\pa{\pa{f_m}_{m=1}^N}(x)
		=\pa{F_N^{-1} \circ A_T }\pa{\pa{\left\{\widehat{f_m}(j)\right\}_{j\in\Z}}_{m=1}^N}(x) \\
		& \quad = F_N^{-1}\pa{\pa{\left\{ \sum_{s=1}^N\sum_{j\in\Z} \widehat{f_s}(j) 
		T_n\pa{\sinc}\pa{N(m-j)-s+1}\right\}_{m\in\Z}}_{n=1}^N}(x) \\
		& \quad = F_N^{-1}\pa{\pa{\left\{ \sum_{s=1}^N\sum_{j\in\Z} \widehat{f_s}(j) 
		\mathcal{F}\pa{\sum_{r\in\Z} T_n(\sinc)(Nr-s+1) e^{2\pi i r \cdot}}(m-j)\right\}_{m\in\Z}}_{n=1}^N}(x) \\
		& \quad = F_N^{-1}\pa{\pa{\left\{ \sum_{s=1}^N \mathcal{F}\pa{f_s(\cdot) 
		\sum_{r\in\Z} T_n(\sinc)(Nr-s+1) e^{2\pi i r \cdot}}(m)\right\}_{m\in\Z}}_{n=1}^N}(x) \\
		& \quad = \pa{\sum_{s=1}^N f_s(x) \pa{\sum_{r\in\Z} T_n(\sinc)(Nr-s+1) e^{2\pi i r x}}}_{n=1}^N \\
		& \quad =\pa{\sum_{r\in\Z} T_n(\sinc)(Nr-s+1) e^{2\pi i r x}}_{n,s=1}^N \pa{f_s(x)}_{s=1}^N.
	\end{align*}
	It is clear that the operator $A_T:\pa{\ell^2}^N\to\pa{\ell^2}^N$ is invertible if and only if 
	$F_N^{-1} \circ A_T \circ F_N:L^2\pa{\R / \Z,\R^N}\to L^2\pa{\R / \Z,\R^N}$ is invertible and, due 
	to what we just obtained, this is equivalent to the existence of a bounded inverse of the matrix of functions
	$$\pa{\sum_{r\in\Z} T_n(\sinc)(Nr-s+1) e^{2\pi i r x}}_{n,s=1}^N.$$
	Using Lemma \ref{lem2} we obtain that, for $x\in \pa{\frac{N-2}{2},\frac{N}{2}}$, we have that
	\begin{align*}
		&\pa{\sum_{r\in\Z} T_n(\sinc)(Nr-s+1) e^{2\pi i r x}}_{n,s=1}^N \\
		&\quad = \pa{\frac{1}{N} \sum_{m=1}^N \widehat{\sinc}\pa{\frac{m-1-x}{N}} K_n\pa{\frac{m-1-x}{N}}
		e^{-\frac{2\pi i \pa{s-1} \pa{m-1-x}}{N}}}_{n,s=1}^N \\
		&\quad = \pa{\frac{1}{N}K_n\pa{\frac{m-1-x}{N}}}_{n,m=1}^N 
        \pa{e^{-\frac{2\pi i (s-1) (m-1-x)}{N}}}_{m,s=1}^N \\
		&\quad = \pa{\mathcal{M}_T(x)}^T \pa{e^{-\frac{2\pi i (s-1) (m-1-x)}{N}}}_{m,s=1}^N 
	\end{align*}
	and as
	\begin{align*}
		\det \pa{e^{-\frac{2\pi i (s-1) (m-1-x)}{N}}}_{m,s=1}^N
		&= \prod_{1\leq n < m \leq N} \pa{e^{-\frac{2\pi i (m-1-x)}{N}}-e^{-\frac{2\pi i (n-1-x)}{N}}} \\
		&= \prod_{1\leq n < m \leq N} e^\frac{2\pi i x}{N} 
		\pa{e^{-\frac{2\pi i (m-1)}{N}}-e^{-\frac{2\pi i (n-1)}{N}}} \neq 0
	\end{align*}
	we deduce that $A_T:\pa{\ell^2}^N\to\pa{\ell^2}^N$ is invertible if and only if $\essinf_{x\in \pa{\frac{N-2}{2},\frac{N}{2}}} \abs{\det \mathcal{M}_T(x)}>0$, by Lemma \ref{lem1}. So now we will show that 2) holds if and only if $A_T:\pa{\ell^2}^N\to\pa{\ell^2}^N$ is invertible. One of the implications is easy to obtain as if 2) holds, then we have that
	$$\pa{\suc{f(Nm+n-1)}_{m\in\Z}}_{n=1}^N
	=\pa{\suc{\sum_{s=1}^N \sum_{j\in\Z} T_s(f)(Nj) g^T_s(N(m-j)+n-1)}_{m\in\Z}}_{n=1}^N$$
	and from this we can deduce that $\pa{A_T}^{-1}:\pa{\ell^2}^N\to\pa{\ell^2}^N$ exists and is defined by
	$$\pa{A_T}^{-1}\pa{\pa{\suc{b^n_m}_{m\in\Z}}_{n=1}^N}
	=\pa{\suc{\sum_{s=1}^N \sum_{j\in\Z} b_j^s \ g^T_s(N(m-j)+n-1)}_{m\in\Z}}_{n=1}^N.$$
	For the other implication, we start by assuming that $A_T:\pa{\ell^2}^N\to\pa{\ell^2}^N$ is invertible, 
	note that its inverse has to satisfy
	$$\pa{A_T}^{-1}\pa{\pa{\left\{ T_n(f)\pa{km}\right\}_{m\in\Z}}_{n=1}^N}
	=\pa{\left\{ f(km+n-1) \right\}_{m\in\Z}}_{n=1}^N,$$
	for all $f\in PW_\pi$. Now, for each $s=1,...,N$, we will consider the elements of $\pa{\ell^2}^N$ defined as
    $$\pa{\suc{C^T_{s,n,m}}_{m\in\Z}}_{n=1}^N:=\pa{A_T}^{-1}\pa{\pa{\suc{\delta_{n,s}\delta_{m,0}}_{m\in\Z}}_{n=1}^N}$$
    and use them to define the following functions in $PW_\pi$
    $$g^T_s(x)= \sum_{n=1}^N \sum_{m\in\Z} C^T_{s,n,m} \sinc (x-Nm-n+1).$$
    Using the linearity of $\pa{A_T}^{-1}$ we can obtain that
    \begin{align*}
		&\pa{\left\{ f(Nm+n-1) \right\}_{m\in\Z}}_{n=1}^N
		= \pa{A_T}^{-1}\pa{\pa{\left\{ T_n(f)(Nm)\right\}_{m\in\Z}}_{n=1}^N} \\
		&\quad = \pa{A_T}^{-1}\pa{\pa{\left\{ 
		\sum_{s=1}^N \sum_{j\in\Z} T_s(f)(Nj) \delta_{n,s}\delta_{m-j,0} \right\}_{m\in\Z}}_{n=1}^N} \\
		&\quad = \sum_{s=1}^N \sum_{j\in\Z} T_s(f)(Nj)
		\pa{A_T}^{-1}\pa{\pa{\left\{ \delta_{n,s}\delta_{m-j,0} \right\}_{m\in\Z}}_{n=1}^N} \\
		&\quad = \sum_{s=1}^N \sum_{j\in\Z} T_s(f)(Nj)
		\pa{\suc{C^T_{s,n,m-j}}_{m\in\Z}}_{n=1}^N
		= \pa{\suc{\sum_{s=1}^N \sum_{j\in\Z} T_s(f)(Nj) 
		C^T_{s,n,m-j}}_{m\in\Z}}_{n=1}^N,
	\end{align*}
	for any function $f\in PW_\pi$. Using this we can finally obtain that
	\begin{align*}
		f(x) &= \sum_{n=1}^N \sum_{m\in\Z} f(Nm+n-1) \sinc(x-Nm-n+1) \\
		&= \sum_{n=1}^N \sum_{m\in\Z} \sum_{s=1}^N \sum_{j\in\Z} T_s(f)(Nj) C^T_{s,n,m-j} \sinc(x-Nm-n+1) \\
		&= \sum_{s=1}^N \sum_{j\in\Z} T_s(f)(Nj) \sum_{n=1}^N \sum_{m\in\Z} C^T_{s,n,m-j} \sinc(x-Nm-n+1) \\
		&= \sum_{s=1}^N \sum_{j\in\Z} T_s(f)(Nj) \sum_{n=1}^N \sum_{m\in\Z} C^T_{s,n,m} \sinc(x-Nj-Nm-n+1) \\
		&= \sum_{s=1}^N \sum_{j\in\Z} T_s(f)(Nj) g^T_s(x-Nj),
	\end{align*}
	for any function $f\in PW_\pi$. So the only thing left to show is that the functions $g^T_n$, $n=1,..,N$, satisfy the characterizations i) and ii). Using the definition of these functions, we get 
	\begin{align*}
		\pa{\suc{T_n\pa{g^T_s}(Nm)}_{m\in\Z}}_{n=1}^N
		&= A_T \pa{\pa{\suc{g^T_s(Nm+n-1)}_{m\in\Z}}_{n=1}^N} \\
		&= A_T \pa{\pa{\suc{C^T_{s,n,m}}_{m\in\Z}}_{n=1}^N}
		= \pa{\suc{\delta_{n,s}\delta_{m,0}}_{m\in\Z}}_{n=1}^N,
	\end{align*}
	showing that these functions satisfy characterization i), the uniqueness comes from the existence of the interpolation formula. Using this together with Lemma \ref{lem2} we obtain
	$$ \delta_{n,s} =\sum_{m\in\Z} T_n\pa{g^T_s}(Nm) e^{2\pi i m x} \\
	= \frac{1}{N} \sum_{m=1}^N \widehat{g^T_s}\pa{\frac{m-1-x}{N}} K_n\pa{\frac{m-1-x}{N}},$$
	for almost every $x\in\pa{\frac{N-2}{2},\frac{N}{2}}$, or equivalently
	$$Id= \pa{\widehat{g^T_s}\pa{\frac{m-1-x}{N}}}_{s,m=1}^N \pa{\frac{1}{N} K_n\pa{\frac{m-1-x}{N}}}_{m,n=1}^N,$$
	for almost every $x\in\pa{\frac{N-2}{2},\frac{N}{2}}$, therefore obtaining characterization ii). \\
\end{proof}

It should be noted that from this proof one can easily deduce Corollary \ref{cor}. Using that
$$A_T\pa{\pa{\left\{ f(Nm+n-1) \right\}_{m\in\Z}}_{n=1}^N}=\pa{\left\{ T_n(f)(Nm)\right\}_{m\in\Z}}_{n=1}^N$$
and that under 1) or 2) $A_T$ is bounded and, by the open mapping theorem, its inverse is also bounded, and consequently
$$\norm{(A_T)^{-1}}^{-1} \sum_{m\in\Z}\abs{f(m)}^2
\leq \sum_{n=1}^N \sum_{m\in\Z} \abs{T_n(f)(Nm)}^2 
\leq \norm{A_T} \sum_{m\in\Z}\abs{f(m)}^2,$$
to end the proof, we simply use that for any $f\in PW_\pi$
$$ \sum_{m\in\Z}\abs{f(m)}^2=\int_\R \abs{f(x)}^2 dx.$$

\subsection{Second proof of Theorem \ref{theo}}
The following proof is a generalization of the ideas used by Vaaler in \cite{Va1985} in the proof of his formula, previously stated in the introduction.

\begin{proof}[Second proof of Theorem \ref{theo}:] \quad
	
	\underline{$1)\implies 2)$:} We start by considering, for $n=1,...,N$, the functions 
    $u^T_n\in L^\infty\pa{-\frac{1}{2},-\frac{1}{2}+\frac{1}{N}}$ that satisfy
    $$\pa{\mathcal{M}_T\pa{-N \xi}}^{-1}=\pa{u^T_n\pa{\xi+\frac{j-1}{N}}}_{n,j=1}^N,$$
    for almost all $\xi\in \pa{-\frac{1}{2},-\frac{1}{2}+\frac{1}{N}}$. From this we obtain 
    that for all $j,s=1,...,N$ we have
    $$\delta_{j,s}=\sum_{n=1}^N \frac{1}{N} K_n\pa{\frac{N\xi+s-1}{N}}u^T_n\pa{\xi+\frac{j-1}{N}},$$
    for almost every $\xi\in \pa{-\frac{1}{2},-\frac{1}{2}+\frac{1}{N}}$. Therefore we get 
    that for all $f\in PW_\pi$ and all $j=1,...,N$ we have
    \begin{align*}
        \widehat{f}\pa{\xi+\frac{j-1}{N}}
        &=\sum_{s=1}^N \widehat{f}\pa{\xi+\frac{s-1}{N}}
        \sum_{n=1}^N \frac{1}{N} K_n\pa{\frac{N\xi+s-1}{N}}u^T_n\pa{\xi+\frac{j-1}{N}} \\
        &= \sum_{n=1}^N \pa{\frac{1}{N}\sum_{s=1}^N \widehat{f}\pa{\frac{s-1-(-N\xi)}{N}} K_n\pa{\frac{s-1-(-N\xi)}{N}}}
        u^T_n\pa{\xi+\frac{j-1}{N}} \\
        &=\sum_{n=1}^N \pa{\sum_{m\in\Z} T_n(f)(Nm) e^{2\pi i m (-N\xi)}} u^T_n\pa{\xi+\frac{j-1}{N}} \\
        &=\sum_{n=1}^N \pa{\sum_{m\in\Z} T_n(f)(Nm) e^{2\pi i m \pa{-N\pa{\xi+\frac{j-1}{N}}}}} u^T_n\pa{\xi+\frac{j-1}{N}},
    \end{align*}
    for almost every $\xi\in \pa{-\frac{1}{2},-\frac{1}{2}+\frac{1}{N}}$, or equivalently
    $$\widehat{f}(\xi)=\sum_{n=1}^N \pa{\sum_{m\in\Z} T_n(f)(Nm) e^{2\pi i m (-N\xi)}} u^T_n(\xi),$$
    for all $f\in PW_\pi$ and almost every $\xi\in\pa{-\frac{1}{2},\frac{1}{2}}$. From this, we 
    easily get that, for all $f\in PW_\pi$, we have
    \begin{align*}
        f(x)&=\int_{-\frac{1}{2}}^\frac{1}{2} \widehat{f}(\xi)e^{2\pi i\xi x}d\xi
        =\int_{-\frac{1}{2}}^\frac{1}{2} 
        \pa{\sum_{n=1}^N \pa{\sum_{m\in\Z} T_n(f)(Nm) e^{2\pi i m (-N\xi)}} u^T_n(\xi)}e^{2\pi i\xi x}d\xi \\
        &=\sum_{n=1}^N \sum_{m\in\Z}T_n(f)(Nm)
        \int_{-\frac{1}{2}}^\frac{1}{2} u^T_n(\xi) e^{2\pi i\xi(x-Nm)}d\xi,
    \end{align*}
    where the convergence holds both in $L^2$ sense and for almost every $x\in\R$, defining 
    $$g^T_n(x):=\int_{-\frac{1}{2}}^\frac{1}{2} u^T_n(\xi) e^{2\pi i\xi x}d\xi$$
    we get the result. \\
    
    \underline{$2)\implies 1)$:} From 3) we can deduce that
    \begin{align*}
        \widehat{f}(\xi)&= \int_{-\infty}^\infty f(x)e^{-2\pi i x \xi}dx
        =\int_{-\infty}^\infty \pa{\sum_{n=1}^N \sum_{m\in\Z} T_n(f)(Nm) g^T_n(x-Nm)} e^{-2\pi i x \xi}dx \\
        & = \sum_{n=1}^N \sum_{m\in\Z} T_n(f)(Nm) \int_{-\infty}^\infty g^T_n(y) e^{-2\pi i (x+Nm) \xi}dx 
        = \sum_{n=1}^N \pa{\sum_{m\in\Z} T_n(f)(Nm) e^{2\pi i m (-N\xi)}} \widehat{g^T_n}(\xi),
    \end{align*}
    for all $f\in PW_\pi$ and almost every $\xi\in \pa{-\frac{1}{2},\frac{1}{2}}$. Note that 
    this is equivalent to saying that for all $f\in PW_\pi$ and all $j=1,...,N$ we have
    \begin{align*}
        &\widehat{f}\pa{\xi+\frac{j-1}{N}}
        =\sum_{n=1}^N \pa{\sum_{m\in\Z} T_n(f)(Nm) e^{2\pi i m \pa{-N\pa{\xi+\frac{j-1}{N}}}}}
        \widehat{g^T_n}\pa{\xi+\frac{j-1}{N}} \\
        &\quad =\sum_{n=1}^N \pa{\sum_{m\in\Z} T_n(f)(Nm) e^{2\pi i m \pa{-N\xi}}}
        \widehat{g^T_n}\pa{\xi+\frac{j-1}{N}} \\
        &\quad =\sum_{n=1}^N \pa{\frac{1}{N}\sum_{s=1}^N
        \widehat{f}\pa{\frac{N\xi+s-1}{N}}K_n\pa{\frac{N\xi+s-1}{N}}}
        \widehat{g^T_n}\pa{\xi+\frac{j-1}{N}} \\
        &\quad =\sum_{s=1}^N \widehat{f}\pa{\xi+\frac{s-1}{N}}
        \pa{\sum_{n=1}^N \frac{1}{N}K_n\pa{\frac{N\xi+s-1}{N}}\widehat{g^T_n}\pa{\frac{N\xi+j-1}{N}}},
    \end{align*}
    for almost every $\xi\in \pa{-\frac{1}{2},-\frac{1}{2}+\frac{1}{N}}$. As this has to hold 
    for all $f\in PW_\pi$, we deduce that for all $j,s=1,...,N$
    $$\delta_{j,s}=\sum_{n=1}^N \frac{1}{N}K_n\pa{\frac{N\xi+s-1}{N}}\widehat{g^T_n}\pa{\frac{N\xi+j-1}{N}},$$
    for almost every $\xi\in \pa{-\frac{1}{2},-\frac{1}{2}+\frac{1}{N}}$, or equivalently that 
    \begin{align*}
    	Id
    	&= \pa{\frac{1}{N}K_n\pa{\frac{N\xi+s-1}{N}}}_{s,n=1}^N
    	\pa{\widehat{g^T_n}\pa{\frac{N\xi+j-1}{N}}}_{n,j=1}^N \\
    	&= \mathcal{M}_T(-N\xi) \cdot \pa{\widehat{g^T_n}\pa{\frac{j-1-(-N\xi)}{N}}}_{n,j=1}^N,
    \end{align*}
    for almost every $\xi\in \pa{-\frac{1}{2},-\frac{1}{2}+\frac{1}{N}}$, or equivalently
    $$Id=\mathcal{M}_T(\xi) \cdot \pa{\widehat{g^T_n}\pa{\frac{j-1-\xi}{N}}}_{n,j=1}^N,$$
    for almost every $\xi\in \pa{\frac{N-2}{2},\frac{N}{2}}$, which gives us characterization ii).
    
    To see that the functions $g^T_n$ also satisfy characterization i) note that, as the right-inverse 
    of a matrix is also its left-inverse, we have
    $$Id=\pa{\widehat{g^T_n}\pa{\frac{s-1-\xi}{N}}}_{n,s=1}^N \pa{\frac{1}{N}K_m\pa{\frac{s-1-\xi}{N}}}_{s,m=1}^N,$$
    for almost every $\xi\in \pa{\frac{N-2}{2},\frac{N}{2}}$. Using this together with Lemma \ref{lem2} we obtain
    \begin{align*}
    	\delta_{n,m} &=\sum_{s=1}^N \widehat{g^T_n}\pa{\frac{s-1-\xi}{N}} \frac{1}{N}K_m\pa{\frac{s-1-\xi}{N}} \\
    	&= \sum_{j\in\Z} T_m\pa{g^T_n}(Nj) e^{2\pi i j \xi},
    \end{align*}
    for all $n,m=1,...,N$ and almost every $\xi\in \pa{\frac{N-2}{2},\frac{N}{2}}$, which implies that
    $$T_m\pa{g^T_n}(Nj)=\delta_{n,m}\delta_{j,0},$$
	for all $m=1,...,N$ and $j\in\Z$. \\
\end{proof}

\subsection{Proof of Theorem \ref{theo2}}

Before starting with the proof of this result, we will make some comments that will allow us to reduce this problem to proving some properties of a linear operator. Note that if we consider the operator
\begin{align*}
	O^{\rho,\delta}_T:  PW_{\delta\pi} &\to \pa{\ell^2}^N \\
	f \quad &\to \pa{\{T_n(f)(\rho k)\}_{k\in\Z}}_{n=1}^N
\end{align*}
then saying that $\pa{\rho\Z,T}$ is a set of stable sampling for $PW_{\delta\pi}$ is equivalent to saying that $O^{\rho,\sigma}_T$ is injective and, similarly saying that $\pa{\rho\Z,T}$ is an interpolation set for $PW_{\delta\pi}$ is equivalent to saying that $O^{\rho,\sigma}_T$ is surjective. Consequently, the fact that $\pa{\rho\Z,T}$ is both a set of stable sampling and an interpolation set for $PW_{\delta\pi}$ is equivalent to the bijectivity of $O^{\rho,\sigma}_T$ which is equivalent to the existence of an interpolation formula in $PW_{\delta\pi}$ depending on $\pa{\{T_n(f)(\rho k)\}_{k\in\Z}}_{n=1}^N$. \\

Now let us consider, for any $a>0$, the operator $H_a$ defined as $H_a(f)(x):=f(ax)$, clearly this operator is invertible, with inverse $H_{1/a}$, and satisfies $H_a(PW_{\delta\pi})=PW_{a\delta\pi}$. It is easy to see that $O^{\rho,\delta}_T\circ H_a=O^{a\rho,\delta/a}_{T^a}$ where $T^\alpha=\pa{T^a_n}_{n=1}^N$ with 
$$T^a_n(f)(x)=\int_\R K_n(a\xi) \widehat{f}(\xi) e^{2\pi i x \xi} d\xi$$
where $K_n$ is the Fourier multiplier associated to $T_n$. Therefore the injectivity or surjectivity of $O^{\rho,\delta}_T$ are equivalent to the same property for $O^{a\rho,\delta/a}_{T^a}$ for any $a>0$. With this we can now start with the proof of the result.

\begin{proof}[Proof of Theorem \ref{theo2}]
	First of all, note that the matrix $M^{\rho}_T$ is simply the one that appeared in Theorem \ref{theo} but associated to the operators $T^{N/\rho}$. We will start by studying the case $\delta=\frac{N}{\delta}$ in each of the cases 1), 2) and 3), from which we will latter deduce the rest of the range of $\delta$. \\
    
    We start by proving 1), we know by Theorem \ref{theo} that the hypothesis in this case is equivalent to the invertibility of the operator $O^{N,1}_{T^{N/\rho}}$ which, by the comments made before, is equivalent to the invertibility of $O^{\rho,N/\rho}_T$. Due to this, in the cases 2) and 3) the operator $O^{\rho,N/\rho}_T$ will not be invertible. Note that, with the reasoning made in the proof of Theorem \ref{theo}, it is easy to see that $O^{\rho,N/\rho}_T$ be injective whenever the matrix $M^\rho_T(x)$ is so, and similarly with the surjectivity. To do this we simply use the first two points of Lemma \ref{lem1}, obtaining that in the case 2) the operator $O^{\rho,N/\rho}_T$ is injective but not surjective and in the case 3) it is not injective nor surjective. \\
    
    For the rest of the range of $\delta$, we will use that if $\delta'\geq \delta$ then $O^{\rho,\delta'}_T\vert_{PW_{\delta\pi}}=O^{\rho,\delta}_T$. From this it is easy to deduce that, if $\delta'>\delta$, we have that if $O^{\rho,\delta'}_T$ is injective then $O^{\rho,\delta}_T$ is also injective but not surjective and, similarly, if $O^{\rho,\delta}_T$ is surjective then $O^{\rho,\delta'}_T$ is also surjective but not injective. \\
\end{proof}

\section{Some applications}

We now present some applications of Theorem \ref{theo}, proving the existence or non-existence of specific interpolation formulas. In the following applications, the operators $T_n$ we will be considering will be derivatives, translations, difference quotients and compositions and linear combinations of such operators.

\subsection{General observations}

The first observation we want to highlight is the case where the operators we consider have a common ``factor", meaning that $T_n=T'\circ T_n^*$ for all $n=1,...,N$. In this case, we will have that

\begin{align*}
    \mathcal{M}_T(x)
    &=\pa{\frac{1}{N} K_n\pa{\frac{m-1-x}{N}}}_{m,n=1}^N
    = \pa{\frac{1}{N} K' \pa{\frac{m-1-x}{N}}K^*_n\pa{\frac{m-1-x}{N}}}_{m,n=1}^N \\
    &= \pa{ \sum_{s=1}^N K' \pa{\frac{s-1-x}{N}} \delta_{m,s} \frac{1}{N}K^*_n\pa{\frac{s-1-x}{N}}}_{m,n=1}^N \\
    &= \pa{K' \pa{\frac{s-1-x}{N}} \delta_{m,s}}_{m,s=1}^N M^{T^*}(x).
\end{align*}
From this we deduce that, if $T_n=T'\circ T_n^*$, an interpolation formula depending on $\suc{T_n(f)(Nm)}_{m\in\Z}$ exists if and only if $T'$ is invertible in $PW_\pi$ and the interpolation formula depending on $\suc{T^*_n(f)(Nm)}_{m\in\Z}$ exists. A simple example of this equivalence is when $T'$ is a translation, in which case we have that an interpolation formula depending on $\suc{T_n(f)(Nm)}_{m\in\Z}$ exists then the one depending on $\suc{T_n(f)(Nm+a)}_{m\in\Z}$, 
for any $a\in\C$, also exists. \\

Another case we want to mention is that if all the multipliers have a common root in 
$\left[-\frac{1}{2},\frac{1}{2}\right]$ then the interpolation formula does not exist. This is because in this case there will exist $m_0\in\{1,...,N\}$ and $x_0\in\left[\frac{N-2}{2},\frac{N}{2}\right]$ such that $K_n\pa{\frac{m_0-1-x_0}{N}}=0$ for all $n=1,...,N$ and therefore the $m_0$-th column of $\mathcal{M}_T(x_0)$ will be all 0, making it not invertible. 

As an example of this, we have that if $T_n(f)=P_n(D)(f)(\cdot+a_n)$ for all $n=1,...,N$ then at least one of the differential polynomials need to have an independent term, because if this is not the case then their associated multipliers, $K_n(x)=P_n(2\pi i x) e^{2\pi i x a_n}$, will have $x=0$ as a common root. Note that this could also be deduced from the previous observation as if none of the polynomials had an independent term then all the operators would have the derivative, an operator that is not invertible, as a common factor.

\subsection{Dynamical sampling}

We consider the case where $T_n= T^{n-1}$ for all $n=1,...,N$, and therefore $K_n(x)=\pa{K(x)}^n$. The matrix in this case has the following form
$$M_{\pa{T^{n-1}}_{n=1}^N}(x)=\pa{\frac{1}{N}\pa{K\pa{\frac{m-1-x}{k}}}^{n-1}}_{m,n=1}^N.$$
The determinant of this matrix is 
$$\det M_{\pa{T^{n-1}}_{n=1}^N}(x)=\frac{1}{N^N} \prod_{0\leq n< m\leq N} \pa{K\pa{\frac{m-1-x}{N}}-K\pa{\frac{n-1-x}{N}}},$$
therefore, by Theorem \ref{theo}, if $K(x)\neq K(y)$ whenever $x,y\in[-\frac{1}{2},\frac{1}{2}]$ are such that $x-y\in \frac{1}{N}\Z$ then 
$$f(x)=\sum_{n=1}^N \sum_{m\in\Z} T^{n-1}(f)(Nm) g^{T,N}_{n}(x-Nm),$$
for all $f\in PW_\pi$\footnote{For instance, if $K$ is injective in $[-1/2,1/2]$, such formula will hold.}. We now proceed to obtain  more explicit formulas for the interpolating functions $g^{T,N}_{n}$. Using characterization i) might be complicated in such a general setting, so we will use characterization ii). Note that computing an explicit inverse of the previously stated matrix is quite difficult, but as we will see, there will be no need to do so in this case. Characterization ii) tells us that
$$\frac{1}{N}\sum_{m=1}^N \widehat{g^{T,N}_m}\pa{\frac{n-1-x}{N}} \pa{K\pa{\frac{j-1-x}{N}}}^{m-1}=\delta_{n,j},$$
for all $x\in \pa{\frac{N-2}{2},\frac{N}{2}}$ and all $n,j=1,...,N$, from this we can deduce that
\begin{equation} \label{DynSam1}
    \frac{1}{N}\sum_{m=1}^N \widehat{g^{T,N}_m}\pa{\frac{n-1-x}{N}} y^{m-1}
    =\prod_{\substack{j=1 \\ j\neq n}}^N \frac{y-K\pa{\frac{j-1-x}{N}}}{K\pa{\frac{n-1-x}{N}}-K\pa{\frac{j-1-x}{N}}},
\end{equation}
for all $x\in \pa{\frac{N-2}{2},\frac{N}{2}}$, all $n=1,...,N$ and all $y\in\R$.

As a consequence of this, we have that
$$\widehat{g^{T,N}_N}\pa{\frac{m-1-x}{N}}
= \frac{N}{\prod_{\substack{j=1 \\ j\neq m}}^N \pa{K\pa{\frac{m-1-x}{N}}-K\pa{\frac{j-1-x}{N}}}},$$
for $m=1,..,N$. For the rest of the interpolating functions we will proceed as follows, first of all, using characterization i) we have that the functions $g^{T,N}_n$ are the only ones satisfying
$$T^{m-1}\pa{g^{T,N}_n}(N s)=\delta_{m,n}\delta_{s,0},$$
for all $n,m=1,...,N$ and $s\in\Z$. Using this we have that for all $m=1,...,N$, $n=2,...,N$ and $s\in\Z$ the following holds
\begin{align*}
    T^{m-1}\pa{T\pa{g^{T,N}_n}}(N s)
    &= T^{(m+1)-1}\pa{g^{T,N}_n}(N s)=
    \begin{cases}
        \delta_{m+1,n} \delta_{s,0} & \text{ if } m\neq N \\
        T^N\pa{g^{T,N}_n}(Ns) & \text{ if } m= N
    \end{cases} \\
    &= \delta_{m,n-1} \delta_{s,0} +\sum_{j\in\Z} T^N\pa{g^{T,N}_n}(Nj) \delta_{m,N} \delta_{s-j,0} \\
    &= T^{m-1}\pa{g^{T,N}_{n-1}}(Ns)+\sum_{j\in\Z} T^N\pa{g^{T,N}_n}(N j) T^{m-1}\pa{g^{T,N}_N}(N(s-j)) \\
    &= T^{m-1}\pa{g^{T,N}_{n-1}+\sum_{j\in\Z} T^N\pa{g^{T,N}_n}(N j) g^{T,N}_N(\cdot-Nj)}(Ns),
\end{align*}
and therefore 
$$g^{T,N}_{n-1}(x)=T\pa{g^{T,N}_n}(x)-\sum_{j\in\Z} T^N\pa{g^{T,N}_n}(N j) g^{T,N}_N(x-Nj).$$

\subsection{Formulas depending on $\suc{f^{(k)}(Nm)}_{m\in\Z}$, $k=0,...,N-1$, for any $N\in\N$} 
As mentioned in the introduction, these formulas were obtained by Littmann in his thesis \cite{Li2003}, but we will now give an alternative proof of this result using our theorem. Note that, this is a particular case for the previous example, where the operator $T$ taken is the derivative, whose associated multiplier is $K(x)=2\pi i x$. Note that this multiplier satisfies the condition stated in the previous example for all $N$, because it is injective, and therefore we have that
$$f(x)=\sum_{n=1}^N \sum_{m\in\Z} f^{(n-1)}(Nm) g^N_n(x-Nm),$$
for all $f\in PW_\pi$ and all $N\in\N$. Using \eqref{DynSam1} to this particular case we obtain
\begin{align}\begin{split}\label{Litt1}
	\frac{1}{N} \sum_{m=1}^N \widehat{g^N_m}\pa{\frac{n-1-x}{N}} y^{m-1}
	&= \prod^N_{\substack{ j=1 \\ j\neq n}} \frac{y-2\pi i \frac{j-1-x}{N}}{2\pi i \frac{n-1-x}{N}-2\pi i \frac{j-1-x}{N}} \\
	&= \pa{\frac{N}{2\pi i}}^{N-1}\frac{1}{\prod^N_{\substack{ j=1 \\ j\neq n}}(n-j)}
	\prod^N_{\substack{ j=1 \\ j\neq n}} \pa{y-2\pi i \frac{j-1-x}{N}} \\
	&= \pa{\frac{N}{2\pi i}}^{N-1}\frac{(-1)^{N-n}}{(n-1)!(N-n)!}
	\prod^N_{\substack{ j=1 \\ j\neq n}} \pa{y+\frac{2\pi i x}{N}-\frac{2\pi i(j-1)}{N}},
\end{split}\end{align}
for all $n=1,...,N$ and $x\in\pa{\frac{N-2}{2},\frac{N}{2}}$. Let us denote
$$P_n^N(y,x)=\prod^N_{\substack{ j=1 \\ j\neq n}} \pa{y+\frac{2\pi i x}{N}-\frac{2\pi i(j-1)}{N}},$$
then by \eqref{Litt1} we have that
\begin{align*}
	\widehat{g^N_m}\pa{\frac{n-1-x}{N}}
	&= N \pa{\frac{N}{2\pi i}}^{N-1}\frac{(-1)^{N-n}}{(n-1)!(N-n)!} \frac{\partial_1 ^{m-1} P_n^N(0,x)}{(m-1)!} \\
	&= N \pa{\frac{N}{2\pi i}}^{N-1}\frac{(-1)^{N-n}}{(n-1)!(N-n)!} \frac{\pa{\frac{N}{2\pi i}}^{m-1}}{(m-1)!}
	\frac{d^{m-1}}{dx^{m-1}} P_n^N (0,x) \\
	&= \frac{\pa{\frac{N}{2\pi i}}^{m-1}}{(m-1)!}
	\frac{d^{m-1}}{dx^{m-1}} \widehat{g^N_1}\pa{\frac{n-1-x}{N}}
	= \frac{\widehat{g^N_1}^{(m-1)}\pa{\frac{n-1-x}{N}}}{(m-1)!(-2\pi i)^{m-1}}
\end{align*}
and as a consequence of this we get
\begin{equation}\label{Litt2}
	\widehat{g^N_{m+1}}(\xi)=-\frac{\widehat{g^N_m}'(\xi)}{2\pi i m},
\end{equation}
for all $m=1,...,N-1$ and almost every $\xi\in\R$. Note that, by \eqref{Litt1}, the functions $\widehat{g^N_m}$ are equal to a polynomial in each of the intervals of the form $\pa{-\frac{1}{2}+\frac{n-1}{N},-\frac{1}{2}+\frac{n}{N}}$, for $n=1,...,N$, but can have discontinuities in the endpoints of such intervals. We can compute the jumps in these discontinuities using \eqref{Litt1} as follows

\begin{align*}
	& \frac{1}{N}\sum_{m=1}^N 
	\pa{\widehat{g^N_m}\pa{\pa{-\frac{1}{2}+\frac{n}{N}}^+}-\widehat{g^N_m}\pa{\pa{-\frac{1}{2}+\frac{n}{N}}^-}} y^{m-1} \\
	& \quad =
	\begin{cases}
		\displaystyle \pa{\frac{N}{2\pi i}}^{N-1} \frac{(-1)^{N-1}}{(N-1)!} 
		\prod_{j=2}^N \pa{y+\pi i-\frac{2\pi i (j-1)}{N}} & \text{ for } n=0 \\
		\displaystyle \pa{\frac{N}{2\pi i}}^{N-1} \frac{(-1)^{N-n-1}}{n!(N-n-1)!}
		\prod_{\substack{j=1 \\ j\neq n+1}}^N \pa{y+\pi i-\frac{2\pi i(j-1)}{N}} \\
		\quad \displaystyle - \pa{\frac{N}{2\pi i}}^{N-1} \frac{(-1)^{N-n}}{(n-1)!(N-n)!}
		\prod_{\substack{j=1 \\ j\neq n}}^N \pa{y+\pi i-\frac{2\pi i j}{N}} & \text{ for } n=1,...,N-1 \\
		\displaystyle -\pa{\frac{N}{2\pi i}}^{N-1} \frac{(-1)^{N-N}}{(N-1)!} 
		\prod_{j=1}^{N-1} \pa{y+\pi i-\frac{2\pi i j}{N}} & \text{ for } n=N
	\end{cases} \\
	& \quad = \pa{\frac{N}{2\pi i}}^{N-1} \frac{(-1)^{N-n-1} N}{n!(N-n)!} \prod_{j=1}^{N-1} \pa{y+\pi i-\frac{2\pi i j}{N}}
	= \pa{\frac{N}{2\pi i}}^{N-1} \frac{(-1)^{N-n-1} N}{n!(N-n)!} \sum_{m=1}^N C^N_m y^{m-1}
\end{align*}
and therefore
\begin{equation}\label{Litt3}
	\widehat{g^N_m}\pa{\pa{-\frac{1}{2}+\frac{n}{N}}^+}-\widehat{g^N_m}\pa{\pa{-\frac{1}{2}+\frac{n}{N}}^-}
	= \pa{\frac{N}{2\pi i}}^{N-1} \frac{(-1)^{N-n-1} N^2}{n!(N-n)!} C^N_m,
\end{equation}
where the coefficients $C^N_m$ are the ones satisfying
$$\prod_{j=1}^{N-1} \pa{y+\pi i-\frac{2\pi i j}{N}}=\sum_{m=1}^N C^N_m y^{m-1}.$$
Combining \eqref{Litt2} and \eqref{Litt3} we obtain

\begin{align}
	& g^N_{m+1}(x)
	= \int_{-\frac{1}{2}}^\frac{1}{2} \widehat{g^N_{m+1}}(\xi) e^{2\pi i x \xi} d\xi
	= -\sum_{n=1}^N \int_{-\frac{1}{2}+\frac{n-1}{N}}^{-\frac{1}{2}+\frac{n}{N}} 
	\frac{\widehat{g^N_m}'(\xi)}{2\pi i m} e^{2\pi i x \xi} d\xi \nonumber \\
	& \quad = \frac{1}{2\pi i m}\sum_{n=1}^N \int_{-\frac{1}{2}+\frac{n-1}{N}}^{-\frac{1}{2}+\frac{n}{N}} 
	\widehat{g^N_m}(\xi) (2\pi i x) e^{2\pi i x \xi} d\xi -\frac{1}{2\pi i m}\sum_{n=1}^N 
	\widehat{g^N_m}\pa{\pa{-\frac{1}{2}+\frac{n}{N}}^-}e^{2\pi i x \pa{-\frac{1}{2}+\frac{n}{N}}}
    \nonumber \\
    & \quad \quad + \frac{1}{2\pi i m}\sum_{n=1}^N 
	\widehat{g^N_m}\pa{\pa{-\frac{1}{2}+\frac{n-1}{N}}^+}e^{2\pi i x \pa{-\frac{1}{2}+\frac{n-1}{N}}} 
    \label{Litt4} \\
	& \quad = \frac{x g^N_m(x)}{m}  + \frac{1}{2\pi i m} \sum_{n=0}^N 
	\pa{\widehat{g^N_m}\pa{\pa{-\frac{1}{2}+\frac{n}{N}}^+}
    -\widehat{g^N_m}\pa{\pa{-\frac{1}{2}+\frac{n}{N}}^-}}
	e^{2\pi i x \pa{-\frac{1}{2}+\frac{n}{N}}} \nonumber \\
	& \quad = \frac{x g^N_m(x)}{m} - \pa{\frac{N}{2\pi i}}^N \frac{C^N_m}{(N-1)! m}
    \sum_{n=0}^N {N \choose n}
	\pa{e^\frac{x\pi i}{N}}^n \pa{-e^{-\frac{x\pi i}{N}}}^{N-n} \nonumber \\
	& \quad = \frac{x g^N_m(x)}{m} - \frac{C^N_m x^N}{(N-1)! m} \pa{\sinc\pa{\frac{x}{N}}}^N, \nonumber
\end{align}
for $m=1,...,N-1$ and all $x\in\R$. Note that that from \eqref{Litt1} it can be easily deduced that
$$\frac{1}{N}\widehat{g^N_N}\pa{\frac{n-1-x}{N}}= \pa{\frac{N}{2\pi i}}^{N-1}\frac{(-1)^{N-n}}{(n-1)!(N-n)!}$$
and therefore
\begin{align*}
    g^N_N (x) &= \int_{-\frac{1}{2}}^\frac{1}{2} \widehat{g^N_N}(\xi) e^{2\pi i \xi x} d\xi 
    = \frac{1}{N} \sum_{n=1}^N \int_\frac{N-2}{2}^\frac{N}{2} \widehat{g^N_N}\pa{\frac{n-1-y}{N}} e^{2\pi i \frac{n-1-y}{N} x} dy \\
    &= \pa{\frac{N}{2\pi i}}^{N-1} \sum_{n=1}^N \frac{(-1)^{N-n} e^\frac{2\pi i (n-1) x}{N}}{(n-1)!(N-n)!} \int_\frac{N-2}{2}^\frac{N}{2} e^{-\frac{2\pi i y x}{N} } dy \\
    &= \pa{\frac{N}{2\pi i}}^{N-1} \sum_{n=1}^N \frac{(-1)^{N-n} e^\frac{2\pi i (n-1) x}{N}}{(n-1)!(N-n)!} \frac{e^{-\pi i x}-e^{-\frac{\pi i (N-2) x}{N} }}{-\frac{2\pi i x}{N}} \\
    &= \pa{\frac{N}{2\pi i}}^{N-1} \frac{e^\frac{x\pi i}{N}-e^{-\frac{x\pi i}{N}}}{\frac{2\pi i x}{N}}
    \frac{e^{-\frac{\pi i (N-1) x}{N}}}{(N-1)!}\sum_{n=1}^N {{N-1}\choose {n-1}} e^\frac{2\pi i (n-1) x}{N} (-1)^{(N-1)-(n-1)} \\
    &= \frac{x^{N-1}}{(N-1)!} \pa{\frac{e^\frac{x\pi i}{N}-e^{-\frac{x\pi i}{N}}}{\frac{2\pi i x}{N}}}^N
    =\frac{x^{N-1}}{(N-1)!} \pa{\sinc\pa{\frac{x}{N}}}^N.
\end{align*}
Combining the explicit formula for $g^N_N$ and \eqref{Litt4} we can conclude that
$$g^N_m(x)=\pa{\sum_{n=m}^N C_n^N \frac{(n-1)!}{(N-1)!} x^{N-n}} \frac{x^{m-1}}{(m-1)!} \pa{\sinc\pa{\frac{x}{N}}}^N,$$
where, as mentioned before, the coefficients $C^N_m$ are the ones satisfying
$$\prod_{j=1}^{N-1} \pa{y+\pi i-\frac{2\pi i j}{N}}=\sum_{m=1}^N C^N_m y^{m-1}.$$

\subsection{Formula depending on $\suc{f(2m+a)}_{m\in\Z}$ and $\suc{f^{(n)}(2m+b)}_{m\in\Z}$}
In this case the operators we are considering are $T_1(f)=f(\cdot+a)$ and $T_2(f)=f^{(n)}(\cdot+b)$, with their corresponding functions $K_1(x)=e^{2\pi i a x}$ and $K_2(x)=(2\pi i x)^n e^{2\pi i b x}$. Therefore, the matrix corresponding to this case is
$$
\begin{pmatrix}
    \displaystyle \frac{1}{2} e^{-\frac{2\pi i x a}{2}} 
    & \displaystyle \frac{1}{2}\pa{-\frac{2\pi i x}{2}}^n e^{-\frac{2\pi i x b}{2}} \\
    \displaystyle\frac{1}{2} e^\frac{2\pi i (1-x) a}{2} 
    & \displaystyle \frac{1}{2}\pa{\frac{2\pi i (1-x)}{2}}^n e^\frac{2\pi i (1-x) b}{2}
\end{pmatrix}=
\begin{pmatrix}
    \displaystyle \frac{1}{2} e^{-a x \pi i} 
    & \displaystyle \frac{(-\pi i)^n}{2} x^n e^{-b x \pi i} \\
    \displaystyle \frac{e^{a\pi i}}{2} e^{-a x \pi i} 
    & \displaystyle \frac{(-\pi i)^n e^{b\pi i}}{2}(x-1)^n e^{-b x \pi i}
\end{pmatrix},$$
that is invertible for all $x\in (0,1)$, if we exclude the cases when $n$ and $a-b$ are even and when $n$ and $a-b$ are odd, with inverse
$$
\begin{pmatrix}
    \displaystyle \frac{2 e^{b\pi i} (x-1)^n e^{a x \pi i}}{(x-1)^n e^{b \pi i} - x^n e^{a \pi i}} 
    & \displaystyle -\frac{2 x^n e^{a x \pi i}}{(x-1)^n e^{b \pi i} - x^n e^{a \pi i}} \\
    \displaystyle -\frac{2 e^{a\pi i} e^{b x \pi i}}{(-\pi i)^n \pa{(x-1)^n e^{b \pi i} - x^n e^{a \pi i}}} 
    & \displaystyle \frac{2 e^{b x \pi i}}{(-\pi i)^n \pa{(x-1)^n e^{b \pi i} - x^n e^{a \pi i}}}
\end{pmatrix}.$$
Therefore, if $n$ and $a-b$ are even or $n$ and $a-b$ are odd, there is no interpolation formula, but otherwise, we have that
$$f(x)=\sum_{m\in\Z}f(2m+a) g_{n,a,b}(x-2m) +\sum_{m\in\Z}f^{(n)}(2m+b) h_{n,a,b}(x-2m),$$
for all $f\in PW_\pi$, where the convergence holds both in $L^2$ sense and uniformly where
\begin{align*}
    g_{n,a,b}(x) 
    &=\frac{1}{2} \int_0^1 \widehat{g_{n,a,b}}\pa{\frac{-\xi}{2}} e^{-x \xi \pi i} d\xi 
    + \frac{1}{2} \int_0^1 \widehat{g_{n,a,b}}\pa{\frac{1-\xi}{2}} e^{x (1-\xi) \pi i} d\xi \\
    &=\int_0^1 \frac{e^{b\pi i} (\xi-1)^n e^{a \xi \pi i}}{(\xi-1)^n e^{b \pi i} - \xi^n e^{a \pi i}} e^{-x \xi \pi i} d\xi 
    -\int_0^1 \frac{\xi^n e^{a \xi \pi i}}{(\xi-1)^n e^{b \pi i} - \xi^n e^{a \pi i}} e^{x (1-\xi) \pi i} d\xi \\
    &= \int_0^1 \frac{\xi^n e^{(x-b) \pi i} - (\xi-1)^n}{\xi^n e^{(a-b) \pi i} - (\xi-1)^n } e^{-(x-a) \xi \pi i} d\xi 
\end{align*}
and
\begin{align*}
    h_{n,a,b}(x)
    &=\frac{1}{2} \int_0^1 \widehat{h_{n,a,b}}\pa{\frac{-\xi}{2}} e^{-x \xi \pi i} d\xi 
    + \frac{1}{2} \int_0^1 \widehat{h_{n,a,b}}\pa{\frac{1-\xi}{2}} e^{x (1-\xi) \pi i} d\xi \\
    &=-\int_0^1 \frac{e^{a\pi i} e^{b \xi \pi i} e^{-x \xi \pi i}}
    {(-\pi i)^n \pa{(\xi-1)^n e^{b \pi i} - \xi^n e^{a \pi i}}} d\xi 
    +\int_0^1 \frac{e^{b \xi \pi i} e^{x (1-\xi) \pi i}}{(-\pi i)^n \pa{(\xi-1)^n e^{b \pi i} - \xi^n e^{a \pi i}}} d\xi \\
    &=\frac{1}{(-\pi i)^n} \int_0^1 \frac{1-e^{(x-a)\pi i}}{\xi^n - (\xi-1)^ne^{(b-a) \pi i}} e^{-(x-b) \xi \pi i} d\xi.
\end{align*}

\subsection{Formula depending on $\suc{f(Nm+a_n)}_{m\in\Z}$, $n=1,...,N$}
These cases correspond to the operators $T_n(f)=f(\cdot+a_n)$, with their associated multipliers 
$K_n(x)=e^{2\pi i a_n x}$, therefore the matrix we are considering is
$$\pa{\frac{e^\frac{2\pi i (m-1-x) a_n}{N}}{N}}_{m,n=1}^N
= \pa{\pa{e^\frac{2\pi i a_s}{N}}^{m-1}}_{m,s=1}^N \pa{\frac{e^{-\frac{2\pi i a_n x}{N}}}{N}\delta_{n,s}}_{s,n=1}^N$$
The second of these matrices is a diagonal matrix whose elements are different from 0 for any $x$ and therefore is invertible, the first one in a Vandermonde matrix, not depending on $x$, and will be invertible if and only if $e^\frac{2\pi i a_n}{N}\neq e^\frac{2\pi i a_m}{N}$ for all $n\neq m$. Therefore, a formula only depending on $\suc{f(Nm+a_n)}_{m\in\Z}$ for $n=1,...,N$ exists if and only if $a_n-a_m\notin N\Z$ for all $n\neq m$, note that the case we are excluding is the one where at least two of the sequences of values we are considering are the same. \\

Now we will use characterization ii) to obtain an explicit interpolation formula. If we denote by $\pa{C_{n,m}}_{n,m=1}^N$ to the inverse of the first of the two matrices of the factorization made before we have that
$$Id=\pa{C_{n,m}}_{n,m=1}^N \pa{\pa{e^\frac{2\pi i a_s}{N}}^{m-1}}_{m,s=1}^N
= \pa{\sum_{m=1}^N C_{n,m} \pa{e^\frac{2\pi i a_s}{N}}^{m-1}}_{m,s=1}^N$$
and therefore
$$\sum_{m=1}^N C_{n,m} y^{m-1}
=\prod_{\substack{s=1 \\ s\neq n}}^N \frac{y-e^\frac{2\pi i a_s}{N}}{e^\frac{2\pi i a_n}{N}-e^\frac{2\pi i a_s}{N}}.$$
From this, we deduce that
$$\pa{\pa{\frac{1}{N}e^{\frac{2\pi i(m-1-x)a_n}{N}}}_{m,n=1}^N}^{-1}
=\pa{N e^\frac{2\pi i a_n x}{N}\delta_{n,s}}_{n,s=1}^N \pa{C_{s,m}}_{s,m=1}^N
=\pa{N e^\frac{2\pi i a_n x}{N} C_{n,m}}_{n,m=1}^N$$
and using characterization ii) we get
\begin{align*}
	g^a_n(x) &=\int_{-\frac{1}{2}}^\frac{1}{2} \widehat{g^a_n}(\xi) e^{2\pi i x \xi} d\xi
	= \frac{1}{N} \sum_{m=1}^N \int_\frac{N-2}{2}^\frac{N}{2}  
	\widehat{g^a_n}\pa{\frac{m-1-y}{N}} e^{2\pi i x \frac{m-1-y}{N}} dy \\
	&= \sum_{m=1}^N \int_\frac{N-2}{2}^\frac{N}{2} e^\frac{2\pi i a_n y}{N} C_{n,m} e^{2\pi i x \frac{m-1-y}{N}} dy
	= \sum_{m=1}^N C_{n,m} e^\frac{2\pi i (m-1) x}{N} 
	\frac{e^{\pi i (a_n-x)}-e^\frac{\pi i (a_n-x) (N-2)}{N}}{\frac{2\pi i (a_n-x)}{N}} \\
	&= \pa{\sum_{m=1}^N C_{n,m} \pa{e^\frac{2\pi i x}{N}}^{m-1}} e^\frac{\pi i (a_n-x)(N-1)}{N} \sinc\pa{\frac{x-a_n}{N}} \\
    &= \sinc\pa{\frac{x-a_n}{N}} e^\frac{\pi i (a_n-x)(N-1)}{N}
    \prod_{\substack{s=1 \\ s\neq n}}^N \frac{e^\frac{\pi i x}{N}}{e^\frac{\pi i a_n}{N}}
    \frac{e^\frac{(x-a_s)\pi i}{N}-e^{-\frac{(x-a_s)\pi i}{N}}}{e^\frac{(a_n-a_s)\pi i}{N}-e^{-\frac{(a_n-a_s)\pi i}{N}}} \\
    &= \sinc\pa{\frac{x-a_n}{N}} \prod_{\substack{s=1 \\ s\neq n}}^N 
    \frac{\sin\pa{\frac{\pi(x-a_s)}{N}}}{\sin\pa{\frac{\pi(a_n-a_s)}{N}}}.
\end{align*}
In conclusion if $N\in\N$ and $a_n-a_m\notin N\Z$ for all $n\neq m$ we have that
$$f(x)=\sum_{n=1}^N \sum_{m\in\Z} f(Nm+a_n) \sinc\pa{\frac{x-Nm-a_n}{N}} \prod_{\substack{s=1 \\ s\neq n}}^N \frac{\sin\pa{\frac{\pi(x-Nm-a_s)}{N}}}{\sin\pa{\frac{\pi(a_n-a_s)}{N}}},$$
for all $f\in PW_\pi$.

\subsection{Formula depending on $\suc{f(2m+a)}_{m\in\Z}$ and $\suc{\frac{f(2m+b-\epsilon)-f(2m+b+\epsilon)}{2\epsilon}}_{m\in\Z}$}
In this case the operators we are considering are $T_1(f)=f(\cdot+a)$ and $T_2(f)=\frac{f(\cdot +b-\epsilon)-f(\cdot +b+\epsilon)}{2\epsilon}$, with their corresponding functions $K_1(x)=e^{2\pi i a x}$ and $K_2(x)=\frac{e^{2\pi i (b-\epsilon) x}-e^{2\pi i (b+\epsilon) x}}{2\epsilon}=2\pi i x \ \sinc(2\epsilon x) e^{2\pi i b x}$. Therefore, the matrix associated to this case is
$$
\begin{pmatrix}
    \frac{1}{2} e^{-\frac{2\pi i x a}{2}} & \frac{1}{2}2\pi i \pa{-\frac{x}{2}}\sinc\pa{-\frac{2\epsilon x}{2}} e^{-\frac{2\pi i x b}{2}} \\
    \frac{1}{2} e^\frac{2\pi i (1-x) a}{2} & \frac{1}{2}2\pi i \pa{\frac{1-x}{2}}\sinc\pa{\frac{2\epsilon (1-x)}{2}} e^\frac{2\pi i (1-x) b}{2}
\end{pmatrix}=
\begin{pmatrix}
    \frac{1}{2} e^{-a x \pi i} & -\frac{x \pi i}{2} \sinc(\epsilon x) e^{-b x \pi i} \\
    \frac{e^{a\pi i}}{2} e^{-a x \pi i} & \frac{(1-x) \pi i e^{b\pi i}}{2} \sinc(\epsilon (1-x)) e^{-b x \pi i}
\end{pmatrix},$$
that will be invertible, with bounded inverse, if and only if the equation 
$$e^{b\pi i} (1-x)\sinc(\epsilon (1-x))+e^{a\pi i} x \ \sinc(\epsilon x)=0$$
has no solution in $x\in [0,1]$. This equation has a solution in three cases: 

1) when $a-b$ is odd, in which case $x=\frac{1}{2}$ is a solution, 

2) when $\epsilon\in\Z\setminus\{0\}$, in which cases $x=0$ and $x=1$ are solutions, 

3) when $a-b$ is even and $\abs{\epsilon}\geq 1$, in which cases $x=\frac{1}{2}-\frac{1}{2\epsilon}$ and $x=\frac{1}{2}+\frac{1}{2\epsilon}$ are solutions. 

Excluding these three cases, the inverse of the previously stated matrix is
$$
\begin{pmatrix}
    \displaystyle \frac{2 e^{a x \pi i} (1-x) \sinc(\epsilon (1-x))}{(1-x) \sinc(\epsilon (1-x))+ e^{(a-b)\pi i} x \ \sinc(\epsilon x)} 
    & \displaystyle \frac{2 e^{a x \pi i} e^{-b\pi i} x \ \sinc(\epsilon x)}{(1-x) \sinc(\epsilon (1-x))+ e^{(a-b)\pi i} x \ \sinc(\epsilon x)} \\
    \displaystyle -\frac{2 e^{b x \pi i}}{\pi i \pa{e^{(b-a)\pi i}(1-x) \sinc(\epsilon (1-x))+ x \ \sinc(\epsilon x)}} 
    & \displaystyle \frac{2 e^{b x \pi i} e^{-a\pi i}}{\pi i \pa{e^{(b-a)\pi i}(1-x) \sinc(\epsilon (1-x))+ x \ \sinc(\epsilon x)}} 
\end{pmatrix}$$
therefore, if we are not in the cases 1), 2) or 3), we have that
$$f(x)=\sum_{m\in\Z} f(2m+a) g_{\epsilon,a,b}(x-2m)+\sum_{m\in\Z} \frac{f(2m+b-\epsilon)-f(2m+b+\epsilon)}{2\epsilon} h_{\epsilon,a,b}(x-2m),$$
where
\begin{align*}
    g_{\epsilon,a,b}(x) 
    &= \frac{1}{2} \int_0^1 \widehat{g_{\epsilon,a,b}}\pa{\frac{-\xi}{2}} e^{-x \xi \pi i} d\xi + \frac{1}{2} \int_0^1 \widehat{g_{\epsilon,a,b}}\pa{\frac{1-\xi}{2}} e^{x (1-\xi) \pi i} d\xi \\
    &= \int_0^1\frac{e^{a \xi \pi i} (1-\xi) \sinc(\epsilon (1-\xi))}{(1-\xi) \sinc(\epsilon (1-\xi))+ e^{(a-b)\pi i} \xi \ \sinc(\epsilon \xi)} e^{-x \xi \pi i} d\xi \\
    & \quad + \int_0^1\frac{e^{a \xi \pi i} e^{-b\pi i} \xi \ \sinc(\epsilon \xi)}{(1-\xi) \sinc(\epsilon (1-\xi))+ e^{(a-b)\pi i} \xi \ \sinc(\epsilon \xi)} e^{x (1-\xi) \pi i} d\xi \\
    &= \int_0^1\frac{(1-\xi) \sinc(\epsilon (1-\xi))+ e^{(x-b)\pi i} \xi \ \sinc(\epsilon \xi)}{(1-\xi) \sinc(\epsilon (1-\xi))+ e^{(a-b)\pi i} \xi \ \sinc(\epsilon \xi)} e^{-(x-a) \xi \pi i} d\xi \\
    &= \int_0^1 \frac{\sin\pa{\pi \epsilon (1-\xi)} + e^{(x-b)\pi i}\sin\pa{\pi \epsilon \xi}}
    {\sin\pa{\pi \epsilon (1-\xi)} + e^{(a-b)\pi i}\sin\pa{\pi \epsilon \xi}} e^{-(x-a)\xi\pi i} d\xi
\end{align*}
and
\begin{align*}
    h_{\epsilon,a,b}(x)
    &=\frac{1}{2} \int_0^1 \widehat{h_{\epsilon,a,b}}\pa{\frac{-\xi}{2}} e^{-x \xi \pi i} d\xi + \frac{1}{2} \int_0^1 \widehat{h_{\epsilon,a,b}}\pa{\frac{1-\xi}{2}} e^{x (1-\xi) \pi i} d\xi \\
    &= -\int_0^1 \frac{e^{b \xi \pi i}}{\pi i \pa{e^{(b-a)\pi i}(1-\xi) \sinc(\epsilon (1-\xi))+ \xi \ \sinc(\epsilon \xi)}}  e^{-x \xi \pi i} d\xi \\
    & \quad + \int_0^1 \frac{e^{b \xi \pi i} e^{-a\pi i}}{\pi i \pa{e^{(b-a)\pi i} (1-\xi) \sinc(\epsilon (1-\xi))+ \xi \ \sinc(\epsilon \xi)}} e^{x (1-\xi) \pi i} d\xi \\
    &= \int_0^1 \frac{e^{(x-a)\pi i}-1}{\pi i \pa{e^{(b-a)\pi i}(1-\xi) \sinc(\epsilon (1-\xi))+ \xi \ \sinc(\epsilon \xi)}}  e^{-(x-b) \xi \pi i} d\xi \\
    &= \int_0^1 \frac{\epsilon \pa{e^{(x-a)\pi i}-1}}{i\pa{e^{(b-a)\pi i}\sin\pa{\pi \epsilon (1-\xi)}-\sin\pa{\pi \epsilon \xi}}}
    e^{-(x-b)\xi \pi i} d\xi.
\end{align*}
Note that, in the three excluded cases, the solutions to the equations are separated points, and therefore a set of measure zero. Therefore, using Theorem \ref{theo2}, we deduce that in any of these cases we have 
$$\norm{f}^2_{L^2(\R)}\lesssim \sum_{m\in\Z} \abs{f(2m+a)}^2 
+ \sum_{m\in\Z}\abs{\frac{f(2m+b-\epsilon)-f(2m+b+\epsilon)}{2\epsilon}}^2 
\lesssim \norm{f}^2_{L^2(\R)}.$$
This can be very easily seen in the case $a=0=b$ and $\epsilon=1$ where, by condition 1) or 2), the interpolation formula does not hold, despite this we have that
$$
\begin{cases}
    f\in PW_\pi \\
    f(2m)=0, \ \forall m\in\Z \\
    \frac{f(2m-1)-f(2m+1)}{2}=0, \ \forall m\in\Z
\end{cases} \implies
\begin{cases}
    f\in PW_\pi \\
    f(m)=0, \ \forall m\in\Z
\end{cases} \implies f=0.$$


\section*{Acknowledgements} The authors are thankful for Carlos Cabrelli for discussions that led to this project. The authors are supported by the project PID2023-146646NB-I00 funded by MICIU/AEI/10.13039/501100011033 and by ESF+, the Basque Government through the BERC 2022-2025 program, and through BCAM Severo Ochoa accreditation CEX2021-001142-S / MICIN / AEI / 10.13039/501100011033. The first author is also suported by the grant PRE2020-094897 from the Spanish Ministry of
Science and Innovation. The second author is also supported by the grant RYC2022-038226-I.

\end{document}